\documentclass[a4paper]{article}
\usepackage[english]{babel}
\usepackage[utf8x]{inputenc}
\RequirePackage{amsmath,amsthm,amssymb,enumitem,graphicx,ifpdf}
\RequirePackage[left=1in,right=1in,top=1in,bottom=1in]{geometry}
\newtheorem{theorem}{Theorem}[subsection]

\newtheorem*{claim*}{Claim}
\newtheorem*{theorem*}{Theorem}

\newtheorem*{lemma*}{Lemma}
\theoremstyle{definition}

\newtheorem*{definition*}{Definition}

\newtheorem*{proposition*}{Proposition}

\newtheorem*{notation*}{Notation}

\newtheorem*{remark*}{Remark}

\numberwithin{equation}{section}
\newcommand{\R}{\mathbb{R}}

\newcommand{\Z}{\mathbb{Z}} 
\newcommand{\N}{\mathbb{N}} 
\renewcommand{\Pr}{\text{Pr}}

\usepackage{amsrefs}

\title{A Comment on the Sums $\sum_{n \in \Z} \frac{(-1)^{nk}}{(an+1)^k}$}
\author{Vivek Kaushik}
\date{}
\begin{document}
\maketitle
\begin{abstract}
We recall a proof of Euler's identity $\sum_{n=1}^{\infty} \frac{1}{n^2}=\frac{\pi^2}{6}$ involving the evaluation of a double integral. We extend the method to find Hurwitz Zeta series of the form $S(k,a)=\sum_{n \in \Z} \frac{(-1)^{nk}}{(an+1)^k},$ where $a \in \N \setminus \lbrace 1 \rbrace$ and  $k \in \N.$ In particular, we consider a general $k$-dimensional integral over $(0,1)^k$ that equals the series representation $S(k,a).$ Then we use an algebraic change of variables that diffeomorphically maps $(0,1)^k$ to a $k$-dimensional hyperbolic polytope. We interpret the integral as a sum of two probabilities, and find explicit representations of such probabilities with combinatorial techniques.  
\end{abstract}
\section{Introduction}

In this article, we evaluate  Hurwitz Zeta Series of the form \begin{align}\label{S(k,a)}
    S(k,a)  &= \sum_{n \in \Z} \frac{(-1)^{nk}}{(an+1)^k}, \quad a \in \N \setminus \lbrace 1 \rbrace, 
    \ k \in \N.
\end{align} 
The values of such series can be obtained through standard techniques from Fourier Analysis and complex variables. Some specific examples of $S(k,a)$ are found in \cites{JC,BFY,VK}, with the case $a=4$ being the focal point of \cites{BCK,NDE,ZS2010,ZS2012}. 

In particular, we provide an alternative method using multiple integration. We generalize the double integral proof of 
\begin{align}\label{Zeta(2)}
\sum_{n \geq 1} \frac{1}{n^2}=\frac{\pi^2}{6} 
\end{align} 
by 
Zagier and Kontsevich \cite[p. 8 - 9]{ZK}, which is as follows. Consider the integral
\begin{align}\label{Zagier Integral}
    I_2 &= \int_{(0,1)^2} \frac{1}{\sqrt{x_1 x_2} (1-x_1x_2)} \ dx_2 \ dx_1.
\end{align}
We convert the integrand into a geometric series, 
\begin{align}\label{S(2,2) Geometric Series}
    \frac{1}{\sqrt{x_1 x_2} (1-x_1x_2)} &= \sum_{n \geq 0} (x_1x_2)^{n-1/2}.
\end{align}
Replacing the integrand with the geometric series representation, we find
\begin{align}
    I_2 &= \int_{(0,1)^2} \sum_{n \geq 0} (x_1x_2)^{n-1/2} \ dx_1 \ dx_2 \nonumber  \\
    &=  \sum_{n \geq 0} \int_{(0,1)^2} (x_1x_2)^{n-1/2} \ dx_1 \ dx_2 \label{MCT} \\
    &= \sum_{n \geq 0} \frac{1}{(n+1/2)^2} \nonumber \\
    &=\sum_{n \geq 0} \frac{4}{(2n+1)^2} \nonumber \\
    &= 4 S(2,2) \nonumber,
\end{align}
where the interchanging of sum and integral in \eqref{MCT} follows from the Monotone Convergence Theorem.
On the other hand, the change of variables
\begin{align}\label{Zagier Kontsevich COV}
x_1=\frac{\xi_1^2(1+\xi_2^2)}{1+\xi_1^2}, \quad x_2=\frac{\xi_2^2(1+\xi_1^2)}{1+\xi_2^2},    
\end{align}
 has Jacobian Determinant
\begin{align}\label{Zagier COV Determinant}
\det \frac{\partial(x_1,x_2)}{\partial (\xi_1,\xi_2)} &= \frac{4\sqrt{x_1x_2}(1-x_1x_2)}{(1+\xi_1^2)(1+\xi_2^2)}
\end{align}
and diffeomorphically maps $(0,1)^2$ to the hyperbolic triangle 
\begin{align}\label{hyperbolic polytope}
\mathbb{H}^2 &= \lbrace (\xi_1,\xi_2) \in \R^2: \xi_1\xi_2 <1, \xi_1, \xi_2>0 \rbrace. 
\end{align}
Hence,
\begin{align*}
I_2 &= \int_{\mathbb{H}^2} \frac{4}{(1+\xi_1^2)(1+\xi_2^2)} \ d\xi_2 \ d\xi_1 \\
&= \int_{0}^{\infty} \frac{4\cot^{-1}(\xi_1)}{1+\xi_1^2} \ d \xi_1 \\
&= \int_{0}^{\infty} \frac{2 \pi - 4\tan^{-1}(\xi_1)}{1+\xi_1^2} \ d \xi_1 \\
&=\frac{\pi^2}{2}.
\end{align*}
Thus, $$S(2,2)=\frac{\pi^2}{8}.$$ Finally, we can write
\begin{align*}
    \sum_{n \geq 1}\frac{1}{n^2} &= \sum_{n \geq 1} \frac{1}{(2n)^{2}} + \sum_{n \geq 0} \frac{1}{(2n+1)^{2}}\\  &=\sum_{n \geq 1} \frac{1}{(2n)^{2}} + S(2,2) \\
    &= \frac{1}{4} \zeta(2) + \frac{\pi^2}{8},
\end{align*}
from which we may deduce $$\sum_{n \geq 1}\frac{1}{n^2}=\frac{\pi^2}{6}$$ using simple algebra. 

We extend this proof to find arbitrary $S(k,a).$ In particular, we evaluate the integral
\begin{align*}
    I_{k,a} &= \frac{1}{a^k}\int_{(0,1)^k} \frac{(x_1 \dots x_k)^{-1+1/a} + (x_1 \dots x_k)^{-1/a}}{1-(-1)^kx_1 \dots x_k} \ dx_1 \ \dots \ dx_k,
\end{align*} the generalization of \eqref{Zagier Integral}
in two ways. The first way will be to convert the integrand into a geometric series and show that it is equal to $S(k,a).$ On the other hand, we will use a change of variables generalizing \eqref{Zagier Kontsevich COV}. 

\section{Evaluation of $I_{k,a}$}
\subsection{From Integral to Series}
We will evaluate
\begin{align}\label{Generalized Zagier Integral}
    I_{k,a} &= \frac{1}{a^k} \int_{(0,1)^k} \frac{(x_1 \dots x_k)^{-1+1/a} +  (x_1 \dots x_k)^{-1/a}}{1-(-1)^k x_1 \dots x_k} \ dx_1 \ \dots \ dx_k,
\end{align} the generalization of \eqref{Zagier Integral}
in two ways. 

\begin{theorem}\label{S(k,a)=I(k,a)}
We have $$I_{k,a}=S(k,a).$$
\end{theorem}
\begin{proof}

First, we convert the integrand into a geometric series as such
\begin{align}\label{Generalized Zagier GS}
\frac{1}{a^k} \frac{(x_1 \dots x_k)^{-1+1/a} +  (x_1 \dots x_k)^{-1/a}}{1-(-1)^kx_1 \dots x_k}  &= \sum_{n \geq 0} \frac{(-1)^{nk}}{a^k} \left((x_1 \dots x_k)^{n-1+1/a}+(x_1 \dots x_k)^{n-1/a}\right).
\end{align}
Replacing this geometric series representation with the integrand in $I_{k,a}$ we obtain
\begin{align}
    I_{k,a} &= \int_{(0,1)^k}  \sum_{n \geq 0} \frac{(-1)^{nk}}{a^k} \left((x_1 \dots x_k)^{n-1+1/a}+(x_1 \dots x_k)^{n-1/a}\right) \ dx_1 \ \dots \ dx_k \nonumber \\ 
    &=   \sum_{n \geq 0} \int_{(0,1)^k} \frac{(-1)^{nk}}{a^k} \left((x_1 \dots x_k)^{n-1+1/a}+(x_1 \dots x_k)^{n-1/a}\right) \ dx_1 \ \dots \ dx_k \label{Monotone Convergence Theorem} \\
    &= \sum_{n \geq 0} \frac{(-1)^{nk}}{a^k(n+1/a)^k}+\sum_{n \geq 0} \frac{(-1)^{nk}}{a^k(n-1/a+1)^k} \nonumber \\
    &=\sum_{n \geq 0} \frac{(-1)^{nk}}{(an+1)^k}+\sum_{n \geq 0} \frac{(-1)^{nk}}{(an+a-1)^k} \nonumber \\
    &=\sum_{n \geq 0} \frac{(-1)^{nk}}{(an+1)^k}+\sum_{n \geq 0} \frac{(-1)^{(n+1)k}}{(-an-a+1)^k} \nonumber \\
    &=\sum_{n \geq 0} \frac{(-1)^{nk}}{(an+1)^k}+\sum_{n \leq -1} \frac{(-1)^{nk}}{(an+1)^k} \nonumber \\
    &=\sum_{n \in \Z} \frac{(-1)^{nk}}{(an+1)^k}=S(k,a) \nonumber,
\end{align}
where the interchanging of sum and integral in \eqref{Monotone Convergence Theorem} follows from the Monotone Convergence Theorem.
\end{proof}

\subsection{From Integral to Hyperbolic Polytope}
Now, we evaluate $I_{k,a}$ directly. We use the change of variables 
\begin{align} \label{Zagier Generalized COV}
    x_i &= \frac{\xi_i^a(1+\xi_{i+1}^a)}{1+\xi_i^a}, \quad i \in \lbrace 1, \dots, k \rbrace.
\end{align}
where we cyclically index mod $k,$ that is, we have $\xi_{k+1}:=\xi_1.$

\begin{theorem}\label{Zagier COV Theorem}
The change of variables from \eqref{Zagier Generalized COV} has Jacobian Determinant
\begin{align*}
    \det \frac{\partial(x_1, \dots, x_k)}{\partial(\xi_1, \dots, \xi_k)} &= \begin{cases}
    a(\xi_1)^{a-1} & k=1 \\ 
    \frac{a^k(\xi_1 \dots \xi_k)^{a-1}}{(1+\xi^a_1) \dots (1+\xi^a_k)} (1-(-1)^k (\xi_1 \dots \xi_k)^{a}) & \text{else}.
    \end{cases}
\end{align*}
and diffeomorphically maps $(0,1)^k$ to the hyperbolic polytope 
\begin{align*}
    \mathbb{H}^k &=\lbrace (\xi_1, \dots, \xi_k) \in \R^k: \xi_i \xi_{i+1}<1, \xi_i>0, \ i \in \lbrace 1, \dots, k \rbrace  \rbrace.
\end{align*}
\end{theorem}
\begin{proof}
The case $k=1$ is trivial. The case $k=2$ recovers the change of variables in \eqref{Zagier Kontsevich COV} from the introduction; one may see that the stated results in \eqref{Zagier COV Determinant} and \eqref{hyperbolic polytope} corroborate the theorem. 

Suppose $k>2.$ Note that 
\begin{align*}
    \frac{\partial x_i}{\partial \xi_j} &= \begin{cases} \frac{a\xi_i^{a-1}(1+\xi_{i+1}^a)}{(1+\xi_{i}^a)^2} & j=i \\  
    \frac{a\xi_i^{a} \xi_{j}^{a-1}}{1+\xi_{i}^a} & j \equiv i+1 \mod k \\
    0 & \text{else}\end{cases}.
\end{align*}
These are the entries of the Jacobian matrix $\frac{\partial(x_1, \dots, x_k)}{\partial(\xi_1, \dots, \xi_k)}$ corresponding to \eqref{Zagier Generalized COV}. Using cofactor expansion along the first row, we find that 
\begin{align*}
    \det \frac{\partial(x_1, \dots, x_k)}{\partial(\xi_1, \dots, \xi_k)} &= \frac{\partial x_1}{\partial \xi_1} \det(A)- \frac{\partial x_2}{\partial \xi_2} \det(B),
\end{align*}
where $A$ and $B$ are $(k-1) \times (k-1)$ matrices with entries
\begin{align*}
    A_{ij} &= \left[ \frac{\partial x_{i+1}}{\partial \xi_{j}} \right]_{i,j \neq 1} \\
    B_{ij} &= \left[ \frac{\partial x_{i+1}}{\partial \xi_{j}} \right]_{i \neq 1 ,j \neq 2}.
\end{align*}
It can be seen that $A$ is lower triangular, and $B$ is upper triangular. Hence their determinants are easy to calculate using cofactor expansions on the top row of $A$ and the bottom row of $B,$ respectively. The result will simplify down to the claimed Jacobian determinant. 

For the second statement, it can be shown that \eqref{Zagier Generalized COV} is a $\mathcal{C}^1$ bijective map from $(0,1)^k$ to $\mathbb{H}^k$ with $\det \frac{\partial(x_1, \dots, x_k)}{\partial(\xi_1, \dots, \xi_k)} \neq 0$ in $\mathbb{H}^k.$  The Inverse Function Theorem guarantees on any local neighborhood in $\mathbb{H}^k,$ we will have $$\frac{\partial(x_1, \dots, x_k)}{\partial(\xi_1, \dots, \xi_k)}^{-1}=\frac{\partial(\xi_1, \dots, \xi_k)}{\partial(x_1, \dots, x_k)}.$$
\end{proof}

\begin{remark*}
When $a=2,$ if we instead were to make the substitution $x_i=\sqrt{\frac{\xi_i^2(1+\xi_{i+1}^2)}{1+\xi_i^2}}$ and then $\xi_i=\tan(u_i)$ will result in us obtaining $$x_i= \frac{\sin(u_i)}{\cos(u_{i+1})},$$ which is Calabi's trigonometric change of variables, considered in all of \cite{BCK,NDE,ZS2010,ZS2012,DA'K,FL,VK}. Hence, we may view \eqref{Zagier Generalized COV} as an algebraic generalization of Calabi's change of variables.
\end{remark*}

Hence, our two theorems and the change of variables formula imply 
\begin{align}
    S(k,a) &= \int_{\mathbb{H}^k} \frac{1+(\xi_1 \dots \xi_k)^{a-2}}{(1+\xi^a_1) \dots (1+\xi^a_k)} \ d \xi_1 \dots \ d \xi_k \nonumber \\
    &=\int_{\mathbb{H}^k} \frac{1}{(1+\xi^a_1) \dots (1+\xi^a_k)} \ d \xi_1 \dots \ d \xi_k + \int_{\mathbb{H}^k}\frac{(\xi_1 \dots \xi_k)^{a-2}}{(1+\xi^a_1) \dots (1+\xi^a_k)} \ d \xi_1 \dots \ d \xi_k \label{S(k,a) hyperbolic polytope Integral}
\end{align}
We wish to evaluate \eqref{S(k,a) hyperbolic polytope Integral} by mimicking the combinatorial analysis used in \cite[p. 592 - 599]{VK}

\subsection{Hyperbolic Polytope and Combinatorics}
We write $[m]:=\lbrace 1, \dots, m \rbrace$ for $m \in \N.$ Let
$\Xi_i$ for $i \in [k]$ be independent and identically distributed with density function
\begin{align*}
    f_{\Xi_i}(\xi_i) &=  \frac{\frac{a}{\pi} \sin \left( \frac{\pi}{a} \right) }{1+\xi_i^a}, \quad \xi_i \geq 0.
\end{align*}
Similarly, let $\Theta_i$ for $i \in [k]$ be independent and identically distributed with density function
\begin{align*}
    f_{\Theta_i}(\theta_i) &=  \frac{\frac{a}{\pi} \sin \left( \frac{\pi}{a} \right) \theta_i^{a-2} }{1+\theta_i^a}, \quad \theta_i \geq 0.
\end{align*}

\begin{theorem}\label{Xi and Theta Density functions}
For each $i \in [k],$ both $f_{\Xi_i}(\xi_i)$ and $f_{\Theta_i}(\theta_i)$ are valid density functions.
\end{theorem}
\begin{proof}
We recall the cumulative distribution function for $\Xi_i$ and $\Theta_i$ are  
\begin{align}
    F_{\Xi_i}(t) &= \int_{0}^t f_{\Xi_i}(\xi_i) \ d \xi_i \label{Xi CDF} \\ 
    F_{\Theta_i}(t) &= \int_{0}^t f_{\Theta_i}(\theta_i)  \ d \theta_i. \label{Theta CDF}
\end{align}
The claim is equivalent to showing $\lim_{t \rightarrow \infty }F_{X_i}(t)=1$ and $\lim_{t \rightarrow \infty } F_{\Theta_i}(t)=1.$ 

According to Gradshteyn and Ryzhik \cite[Section 2.142]{GR}, we see 
\begin{align} \label{GR Closed Form}
     \int \frac{1}{1+x^a} \ dx  &=  \begin{cases} 
     -\frac{2}{a} \sum_{j=0}^{\lfloor a/2-1 \rfloor} P_j(x) \cos \left( \frac{2j+1}{a} \pi \right) + Q_j(x) \sin \left( \frac{2j+1}{a} \pi \right) & a \text{ even} \\
     \frac{1}{a} \log(1+x) -\frac{2}{a} \sum_{j=0}^{\lfloor a/2-1 \rfloor} P_j(x) \cos \left( \frac{2j+1}{a} \pi \right) + Q_j(x) \sin \left( \frac{2j+1}{a} \pi \right) & a \text{ odd}. \\
     \end{cases},
\end{align}
where 
\begin{align*}
    P_j(x) &= \frac{1}{2} \log \left( x^2-2x \cos \left( \frac{2j+1}{a} \pi \right) +1 \right) \\
    Q_j(x) &= \arctan \left( \frac{x-\cos \left( \frac{2j+1}{a} \pi \right)}{\sin \left( \frac{2j+1}{a} \pi \right)}\right).
\end{align*}
It can be shown upon plugging in $x=0,$ and converting the cosines and sines into complex exponentials, that the right hand side of \eqref{GR Closed Form} is $-\frac{\pi}{a} \csc(\pi/a).$ It also can be shown that as $x \rightarrow \infty,$ the right hand side of \eqref{GR Closed Form} approaches $0.$ Observing these facts will allow us to deduce $\lim_{t \rightarrow \infty }F_{X_i}(t)=1.$

The second result follows from making the substitution $\theta_i \mapsto 1/\theta_i$ in the defining integral representation presented in \eqref{Theta CDF} and deducing $F_{\Theta_i}(t)=1-F_{\Xi_i}(1/t).$ 
\end{proof}

Our main goal is to evaluate
\begin{align}
     \Pr(\Xi_i \Xi_{i+1}<1, i \in [k]) + \Pr(\Theta_i \Theta_{i+1}<1, i \in [k]) \label{hyperbolic polytope Probability},
\end{align}
where $\Xi_{k+1}:=\Xi_1,\Theta_{k+1}:=\Theta_1.$ 
In words, \eqref{hyperbolic polytope Probability} is the sum of the probability that all $\Xi_i$ have cyclically consecutive products less than $1$ and the probability that all $\Theta_i$ have cyclically consecutive products less than $1.$ It is easy to see through \eqref{S(k,a) hyperbolic polytope Integral} that $S(k,a)$ is precisely the product of $\left(\frac{\pi}{a} \csc \left( \frac{\pi}{a} \right) \right)^k$ and \eqref{hyperbolic polytope Probability}. 

We begin with the easy case in calculating \eqref{hyperbolic polytope Probability}. 
\begin{theorem}\label{Prob Xi_i, Theta_i<1}
Suppose $\Xi_i,\Theta_i<1$ for each $i \in [k].$ Then \eqref{hyperbolic polytope Probability} is equal to 
\begin{align*}
    \left(\int_{0}^1 \frac{\frac{a}{\pi} \sin \left(\frac{\pi}{a} \right)}{1+\xi^a} \  d \xi\right)^k + \left(\int_{0}^1 \frac{\frac{a}{\pi} \sin \left (\frac{\pi}{a} \right) \theta^{a-2}}{1+\theta^a} \  d \theta\right)^k
\end{align*}
\end{theorem}
\begin{proof}
Clearly for any $i \in [k],$ we have $\Xi_i\Xi_{i+1},\Theta_i\Theta_{i+1}<1.$ Hence \eqref{hyperbolic polytope Probability} is equal to
\begin{align*}
    \int_{(0,1)^k} f_{\Xi_1}(\xi_1) \dots f_{\Xi_k}(\xi_k)  \ d \xi_1 \dots \ d \xi_k + \int_{(0,1)^k} f_{\Theta_1}(\theta_1) \dots f_{\Theta_k}(\theta_k)  \ d \theta_1 \dots \ d \theta_k,
\end{align*}
from which the result immediately follows.
\end{proof}

The nontrivial case is when there exists $i \in [k]$ such that $\Xi_i, \Theta_i \geq 1.$ We wish to set up an explicit integral representation of \eqref{hyperbolic polytope Probability} in this case. 

\begin{theorem}\label{Conjecture on n}
Suppose $\Xi_1, \dots, \Xi_k, \Theta_1, \dots, \Theta_k$ satisfy the conditions as described by their respective probability terms in \eqref{hyperbolic polytope Probability}. 
Suppose further $r_1, \dots, r_n \in [k]$ are distinct with $1 \leq \Xi_{r_n} \leq  \dots  \leq  \Xi_{r_1}$ and $1 \leq \Theta_{r_n} \leq  \dots  \leq  \Theta_{r_1}.$ Then for distinct $i,j \in [n],$ we have $r_i$ and $r_j$ are pairwise cyclically nonconsecutive, that is,  $|r_i-r_j| \notin \lbrace 1, k-1 \rbrace.$ In addition, $n \leq \lfloor k/2 \rfloor.$ 
\end{theorem}
\begin{proof}
The proofs of these statements are identical to those in \cite[Theorem 3.2, 3.3]{VK}.
\end{proof}

We now use a mechanism to set up the integral corresponding to \eqref{hyperbolic polytope Probability} if $r_1, \dots, r_n \in [k]$ satisfy the first statement of the previous theorem. For each $j \in [n],$ define $\alpha_j$ to be the number of $\Xi_z$ (or $\Theta_z$) from $\lbrace \Xi_{r_j \pm 1} \rbrace$ (or $\lbrace \Theta_{r_j \pm 1} \rbrace$) with the property that $\sup(\Xi_z)=1/\Xi_{r_j}$ (or $\sup(\Theta_z)=1/\Theta_{r_j}$).
In words, $\alpha_j$ counts the number of bounds of the form $0<\Xi_z<1/\Xi_{r_j}$ (or $0<\Theta_z<1/\Theta_{r_j}$) that will appear when we set up the integral for the first probability term (or second probability term) in \eqref{hyperbolic polytope Probability}. 
\begin{theorem}\label{Alpha_j}
We have $$\alpha_j=2-\delta(k,2)-\sum_{m=0}^{j-1}\delta(|r_m-r_j|,2)+\delta(|r_m-r_j|,k-2),$$ where $\delta(a,b)=1$ if $a=b$ and $0$ else.
\end{theorem}
\begin{proof}
The proof is identical to that of \cite[Theorem 3.5]{VK}.
\end{proof}

Now we are ready to set up an integral representation for \eqref{hyperbolic polytope Probability} if $r_1, \dots, r_k \in [k]$ satisfy the first statement in \textbf{Theorem \ref{Conjecture on n}}. 

\begin{theorem}\label{Integral Representation}
If $r_1, \dots, r_k \in [k]$ satisfy the first statement in \textbf{Theorem \ref{Conjecture on n}}, we have \eqref{hyperbolic polytope Probability} is equal to the sum of the integrals 
\begin{align*}
J_{r_1, \dots, r_n} &= \left( \psi(1) \right)^{k-n-\sum_{j=1}^n \alpha_j} \int_{1 \leq \xi_{r_n} \leq \dots \leq \xi_{r_1}} \frac{\left( \frac{\pi}{a} \sin \left( \frac{\pi}{a} \right) \right)^n \left(\psi\left( \frac{1}{\xi_{r_1}} \right) \right)^{\alpha_1} \dots \left(\psi\left( \frac{1}{\xi_{r_n}} \right) \right)^{\alpha_n}}{(1+\xi_{r_1}^a) \dots (1+\xi_{r_n}^a) } \ \ d\xi_{r_n} \  \dots \  d\xi_{r_1} \\
K_{r_1, \dots, r_n} &= \left(\phi(1) \right)^{k-n-\sum_{j=1}^n \alpha_j} \int_{1 \leq \theta_{r_n} \leq \dots \leq \theta_{r_1}} \frac{\left( \frac{\pi}{a} \sin \left( \frac{\pi}{a} \right) \right)^n \left(\phi\left( \frac{1}{\theta_{r_1}} \right) \right)^{\alpha_1} \theta_{r_1}^{a-2} \dots \left(\phi\left( \frac{1}{\theta_{r_n}} \right) \right)^{\alpha_n} \theta_{r_n}^{a-2}}{(1+\theta_{r_1}^a) \dots (1+\theta_{r_n}^a) } \ d\theta_{r_n} \  \dots \  d\theta_{r_1},
\end{align*}
where $\psi(t)$ and $\phi(t)$ are the cumulative distribution functions defined in \eqref{Xi CDF}, \eqref{Theta CDF}, respectively. 
\end{theorem}
\begin{proof}
We already know the integral bounds for $\Xi_{r_1}, \dots, \Xi_{r_n}.$ We already know there are $\sum_{j=1}^n \alpha_j$ bounds of the form $0<\Xi_z<1/\Xi_{r_j}.$ This means there are $k-n-\sum_{j=1}^n \alpha_j$ bounds of the form $0<U_t<1.$ Explicitly, the first probability term in \eqref{hyperbolic polytope Probability} is
\begin{align} \label{Expanded Xi_i>1 Integral}
    \int_{1 \leq \xi_{r_n }\leq \dots \leq \xi_{r_1}}\underbrace{\int_{0}^{1/\xi_{r_1}} \dots \int_{0}^{1/\xi_{r_1}}}_{\alpha_1 \text{ times}} \dots \underbrace{\int_{0}^{1/\xi_{r_n}} \dots \int_{0}^{1/\xi_{r_n}}}_{\alpha_n \text{ times}} \underbrace{\int_{0}^{1} \dots \int_{0}^{1}}_{k-n-\sum_{j=1}^n \alpha_j \text{ times}} f_{\Xi_1}(\xi_1) \dots f_{\Xi_k}(\xi_k) \ dV,
\end{align}
where $dV$ is the product of the differentials $d \xi_1, \dots d \xi_k$ in the appropriate order as dictated by the integral bounds. It follows that \eqref{Expanded Xi_i>1 Integral} is equal to $J_{r_1, \dots, r_n}$ upon evaluating the innermost integrals. 

A similar argument can be used to show that the second probability in \eqref{hyperbolic polytope Probability} is equal to $K_{r_1, \dots, r_n}.$
\end{proof}

Our theorems and \eqref{S(k,a) hyperbolic polytope Integral} give the following result
\begin{align}\label{Gigantic Formula}
    \sum_{n \geq 0} \frac{(-1)^{nk}}{(an+1)^k} &= \left(\int_{0}^1 \frac{\frac{a}{\pi} \sin \left(\frac{\pi}{a} \right)}{1+\xi^a} \  d \xi\right)^k + \left(\int_{0}^1 \frac{\frac{a}{\pi} \sin \left (\frac{\pi}{a} \right) \theta^{a-2}}{1+\theta^a} \  d \theta\right)^k + \sum_{n=1}^{\lfloor k/2 \rfloor} \sum_{\substack{(r_1, \dots, r_n) \in [k]^n: \\ |r_i-r_j| \notin \lbrace 0,1,k-1 \rbrace, \\ 
    i \neq j \in [n]}} J_{r_1, \dots, r_n} + K_{r_1, \dots, r_n},
\end{align}
where $J_{r_1, \dots, r_n}$ and $K_{r_1, \dots, r_n}$ are defined as in the previous theorem.


\begin{bibdiv}
\begin{biblist}

\bib{BFY}{article}{
author = {Bourgade, Paul},
author={Fujita, Takahiko},
author={Yor, Marc},
doi = {10.1214/ECP.v12-1244},
journal = {Electronic Communications in Probability},
pages = {73--80},
publisher = {The Institute of Mathematical Statistics and the Bernoulli Society},
title = {Euler's formulae for $\zeta(2n)$ and products of Cauchy
 variables},
 url = {http://dx.doi.org/10.1214/ECP.v12-1244},
	volume = {12},
	year = {2007}
}

\bib{JC}{article}{
  title={Evaluation of Certain Alternating Series},
  author={Choi, Junesang},
  journal={Honam Mathematical J},
  volume={36},
  pages={263--273},
  year={2014}
}


\bib{BCK}{article}{
  title={Sums of generalized harmonic series and volumes},
  author={Beukers, Frits},
  author={Calabi, Eugenio},
  author={Kolk, Johan AC},
  year={1993},
  journal = {Nieuw Archief voor Wiskunde},
  number = {11},
  pages = {561-573},
  publisher={Rijksuniversiteit Utrecht. Mathematisch Instituut}
}

\bib{FL}{article}{
  title={New definite integrals and a two-term dilogarithm identity},
  author={Lima, FMS},
  journal={Indagationes Mathematicae},
  volume={23},
  number={1},
  pages={1--9},
  year={2012},
  publisher={Elsevier}
}

\bib{DA'K}{article}{
  title={$\zeta(n)$ via hyperbolic functions},
  author={D'Avanzo, Joseph},
  author={Krylov, Nikolai},
  journal={Involve, a Journal of Mathematics},
  volume={3},
  number={3},
  pages={289--296},
  year={2010},
  publisher={Mathematical Sciences Publishers}
}

\bib{NDE}{article}{
 ISSN = {00029890, 19300972},
 URL = {http://www.jstor.org/stable/3647742},
 author = {Elkies, Noam David},
 journal = {The American Mathematical Monthly},
 number = {7},
 pages = {561-573},
 publisher = {Mathematical Association of America},
 title = {On the Sums $\sum_{k=-\infty}^{\infty}(4k + 1)^{-n}$ },
 volume = {110},
 year = {2003}
}

\bib{ZS2010}{article}{
  title={Sums of generalized harmonic series for kids from five to fifteen},
  author={Silagadze, Zurab},
  journal={arXiv preprint arXiv:1003.3602},
  year={2010}
}

\bib{ZS2012}{article}{
  title={Comment on the sums $S(n)=\sum\limits_{k=-\infty}^\infty  \frac{1}{(4k+1)^n}$},
  author={Silagadze, Zurab},
  journal={Georgian Math. J.},
  volume={19},
  number={3},
  pages={587--595},
  year={2012}
} 


\bib{ZK}{book}{
  title={Periods},
  author={Zagier, Don},
  author={Kontsevich, Maxim},
  booktitle={Mathematics Unlimited: 2001 and beyond},
  pages={771--808},
  year={2001},
  publisher={Springer}
}

\bib{GR}{book}{
  title = {Table of Integrals, Series, and Products
},
  author = {Gradshteyn, I. S.}, 
  author={Ryzhik, I. M.},
  edition = {Seventh},
  editor = {Zwillinger, Daniel and Moll, Victor H.},
  isbn = {9780123849342 0123849349},
  publisher = {Academic Press},
  year = {2007}
}

\bib{VK}{article}{
    author = {Kaushik, Vivek},
    author={Ritelli, Daniele},
     TITLE = {Evaluation of harmonic sums with integrals},
   JOURNAL = {Quart. Appl. Math.},
  FJOURNAL = {Quarterly of Applied Mathematics},
    VOLUME = {76},
      YEAR = {2018},
    NUMBER = {3},
     PAGES = {577--600},
      ISSN = {0033-569X},
       DOI = {10.1090/qam/1499},
       URL = {https://doi.org/10.1090/qam/1499},
}


\end{biblist}
\end{bibdiv}
\end{document}